\documentclass[11pt]{amsart}

\usepackage[all]{xy}
\usepackage{mathrsfs}
\usepackage{amsmath, amsthm, amssymb, amscd, amsgen,amsxtra,amsfonts, amsbsy}
\allowdisplaybreaks
\usepackage{verbatim}

\usepackage{tikz}
\usepackage{tikz-cd}
\usepackage{color}
\definecolor{shadecolor}{gray}{0.875}
\definecolor{dblue}{rgb}{0,0,.6}
\usepackage[colorlinks=true, linkcolor=dblue, citecolor=dblue, filecolor = dblue, menucolor = dblue, urlcolor = dblue]{hyperref}
\usepackage{mathtools}
\usepackage{extarrows}

\usepackage{graphicx}
\usepackage{subcaption}

\newcommand{\mathds}[1]{{\mathbb #1}}

\numberwithin{equation}{section}

\begin{document}
%
%
%
\theoremstyle{definition}
\newtheorem{Definition}{Definition}[section]
\newtheorem*{Definitionx}{Definition}
\newtheorem{Convention}{Definition}[section]
\newtheorem{Construction}{Construction}[section]
\newtheorem{Example}[Definition]{Example}
\newtheorem{Examples}[Definition]{Examples}
\newtheorem{Remark}[Definition]{Remark}
\newtheorem*{Remarkx}{Remark}
\newtheorem{Remarks}[Definition]{Remarks}
\newtheorem{Caution}[Definition]{Caution}
\newtheorem{Conjecture}[Definition]{Conjecture}
\newtheorem*{Conjecturex}{Conjecture}
\newtheorem{Question}[Definition]{Question}
\newtheorem{Hypothesis}[Definition]{Hypothesis}
\newtheorem*{Questionx}{Question}
\newtheorem*{Acknowledgements}{Acknowledgements}
\newtheorem*{Notation}{Notation}
\newtheorem*{Organization}{Organization}
\newtheorem*{Disclaimer}{Disclaimer}
\theoremstyle{plain}
\newtheorem{Theorem}[Definition]{Theorem}
\newtheorem*{Theoremx}{Theorem}

\newtheorem*{thmA}{Theorem A}
\newtheorem*{thmB}{Theorem B}

\newtheorem{Proposition}[Definition]{Proposition}
\newtheorem*{Propositionx}{Proposition}
\newtheorem{Lemma}[Definition]{Lemma}
\newtheorem{Corollary}[Definition]{Corollary}
\newtheorem*{Corollaryx}{Corollary}
\newtheorem{Fact}[Definition]{Fact}
\newtheorem{Facts}[Definition]{Facts}
\newtheoremstyle{voiditstyle}{3pt}{3pt}{\itshape}{\parindent}%
{\bfseries}{.}{ }{\thmnote{#3}}%
\theoremstyle{voiditstyle}
\newtheorem*{VoidItalic}{}
\newtheoremstyle{voidromstyle}{3pt}{3pt}{\rm}{\parindent}%
{\bfseries}{.}{ }{\thmnote{#3}}%
\theoremstyle{voidromstyle}
\newtheorem*{VoidRoman}{}

\newenvironment{specialproof}[1][\proofname]{\noindent\textit{#1.} }{\qed\medskip}
\newcommand{\blowup}{\rule[-3mm]{0mm}{0mm}}
\newcommand{\cal}{\mathcal}
\newcommand{\Aff}{{\mathds{A}}}
\newcommand{\BB}{{\mathds{B}}}
\newcommand{\CC}{{\mathds{C}}}
\newcommand{\EE}{{\mathds{E}}}
\newcommand{\FF}{{\mathds{F}}}
\newcommand{\GG}{{\mathds{G}}}
\newcommand{\HH}{{\mathds{H}}}
\newcommand{\NN}{{\mathds{N}}}
\newcommand{\ZZ}{{\mathds{Z}}}
\newcommand{\PP}{{\mathds{P}}}
\newcommand{\QQ}{{\mathds{Q}}}
\newcommand{\RR}{{\mathds{R}}}
\newcommand{\Liea}{{\mathfrak a}}
\newcommand{\Lieb}{{\mathfrak b}}
\newcommand{\Lieg}{{\mathfrak g}}
\newcommand{\Liem}{{\mathfrak m}}
\newcommand{\ideala}{{\mathfrak a}}
\newcommand{\idealb}{{\mathfrak b}}
\newcommand{\idealg}{{\mathfrak g}}
\newcommand{\idealm}{{\mathfrak m}}
\newcommand{\idealp}{{\mathfrak p}}
\newcommand{\idealq}{{\mathfrak q}}
\newcommand{\idealI}{{\cal I}}
\newcommand{\lin}{\sim}
\newcommand{\num}{\equiv}
\newcommand{\dual}{\ast}
\newcommand{\iso}{\cong}
\newcommand{\homeo}{\approx}
\newcommand{\mm}{{\mathfrak m}}
\newcommand{\pp}{{\mathfrak p}}
\newcommand{\qq}{{\mathfrak q}}
\newcommand{\rr}{{\mathfrak r}}
\newcommand{\pP}{{\mathfrak P}}
\newcommand{\qQ}{{\mathfrak Q}}
\newcommand{\rR}{{\mathfrak R}}
\newcommand{\OO}{{\cal O}}
\newcommand{\numero}{{n$^{\rm o}\:$}}
\newcommand{\mf}[1]{\mathfrak{#1}}
\newcommand{\mc}[1]{\mathcal{#1}}
\newcommand{\into}{{\hookrightarrow}}
\newcommand{\onto}{{\twoheadrightarrow}}
\newcommand{\Spec}{{\rm Spec}\:}
\newcommand{\BigSpec}{{\rm\bf Spec}\:}
\newcommand{\Spf}{{\rm Spf}\:}
\newcommand{\Proj}{{\rm Proj}\:}
\newcommand{\Pic}{{\rm Pic }}
\newcommand{\Mov}{{\rm Mov }}
\newcommand{\Nef}{{\rm Nef }}
\newcommand{\MW}{{\rm MW }}
\newcommand{\Br}{{\rm Br}}
\newcommand{\NS}{{\rm NS}}
\newcommand{\Amp}{{\rm Amp}}
\newcommand{\Sym}{\operatorname{Sym}}
\newcommand{\Aut}{{\rm Aut}}
\newcommand{\Autp}{{\rm Aut}^p}
\newcommand{\ord}{{\rm ord}}
\newcommand{\coker}{{\rm coker}\,}
\newcommand{\divisor}{{\rm div}}
\newcommand{\Def}{{\rm Def}}
\newcommand{\rank}{\mathop{\mathrm{rank}}\nolimits}
\newcommand{\Ext}{\mathop{\mathrm{Ext}}\nolimits}
\newcommand{\EXT}{\mathop{\mathscr{E}{\kern -2pt {xt}}}\nolimits}
\newcommand{\Hom}{\mathop{\mathrm{Hom}}\nolimits}
\newcommand{\HOM}{\mathop{\mathscr{H}{\kern -3pt {om}}}\nolimits}
\newcommand{\chari}{\mathop{\mathrm{char}}\nolimits}
\newcommand{\ch}{\mathop{\mathrm{ch}}\nolimits}
\newcommand{\CH}{\mathop{\mathrm{CH}}\nolimits}
\newcommand{\supp}{\mathop{\mathrm{supp}}\nolimits}
\newcommand{\codim}{\mathop{\mathrm{codim}}\nolimits}
\newcommand{\calA}{\mathscr{A}}
\newcommand{\calH}{\mathscr{H}}
\newcommand{\calL}{\mathscr{L}}
\newcommand{\calM}{\mathscr{M}}
\newcommand{\bcalM}{\overline{\mathscr{M}}}
\newcommand{\calN}{\mathscr{N}}
\newcommand{\calX}{\mathscr{X}}
\newcommand{\calK}{\mathscr{K}}
\newcommand{\calD}{\mathscr{D}}
\newcommand{\calY}{\mathscr{Y}}
\newcommand{\calC}{\mathscr{C}}
\newcommand{\piet}{{\pi_1^{\rm \acute{e}t}}}
\newcommand{\Het}[1]{{H_{\rm \acute{e}t}^{{#1}}}}
\newcommand{\Hfl}[1]{{H_{\rm fl}^{{#1}}}}
\newcommand{\Hcris}[1]{{H_{\rm cris}^{{#1}}}}
\newcommand{\HdR}[1]{{H_{\rm dR}^{{#1}}}}
\newcommand{\hdR}[1]{{h_{\rm dR}^{{#1}}}}
\newcommand{\loc}{{\rm loc}}
\newcommand{\et}{{\rm \acute{e}t}}
\newcommand{\defin}[1]{{\bf #1}}
\newcommand{\blue}{\textcolor{blue}}
\newcommand{\red}{\textcolor{red}}
\newcommand{\magenta}{\textcolor{magenta}}

\renewcommand{\HH}{{\rm{H}}}

\title{Curves of maximal moduli on K3 surfaces}

\author{Xi Chen}
\address{632 Central Academic Building, University of Alberta, Edmonton, Alberta T6G 2G1, Canada}
\email{xichen@math.ualberta.ca}

\author{Frank Gounelas}
\address{Georg-August-Universit\"at G\"ottingen, Fakult\"at f\"ur Mathematik und Informatik, Bunsenstr. 3-5, 37073 G\"ottingen, Germany}
\email{gounelas@mathematik.uni-goettingen.de}

\date{\today}

\begin{abstract}
We prove that if $X$ is a complex projective K3 surface and $g>0$, then there exist infinitely many families of
curves of geometric genus $g$ on $X$ with maximal, i.e., $g$-dimensional, variation in moduli. In particular every K3
surface contains a curve of geometric genus 1 which moves in a non-isotrivial family. This implies a conjecture of
Huybrechts on constant cycle curves and gives an algebro-geometric proof of a theorem of Kobayashi that a
K3 surface has no global symmetric differential forms.
\end{abstract}

\maketitle

\setcounter{tocdepth}{1}
\tableofcontents

\section{Introduction}

Building on the work of many people \cite{morimukai,chen,btdensity,bht,liliedtke}, it was recently proved in
\cite{regenerationinfinite} that for any integer $g\geq0$ and any complex projective K3 surface $X$, there is an
infinite sequence of
integral curves $C_n\subset X$ of geometric genus $g\ge 0$ such that for any ample divisor $H$ \[\lim_{n\to\infty} HC_n = \infty.\]
The aim of this paper is to strengthen and give a new proof of this result for curves of genus $g>0$, assuming
only the case $g=0$, and then to derive a number of applications to the geometry of K3 surfaces. In particular we prove
the following.

\begin{thmA}\label{thmA}
Let $X$ be a K3 surface over an algebraically closed field of characteristic zero and $g>0$ an integer. There
exists a sequence of integral curves $C_n\subset X$ of geometric genus $g$, such that 
\[\lim_{n\to\infty} C_n^2 = \infty\] 
and the normalisation morphism of each $C_n$ deforms in a family of smooth genus $g$ curves on $X$ of maximal moduli. 
\end{thmA}

More precisely, for each such $C_n\subset X$ there exists a diagram 
\[
\xymatrix{
    \mathcal{C}_n\ar[d]^{f_n}\ar[r]^{F_n} & X \\
    T_n\ar[r]^{\phi_n} & \calM_g 
}
\]
where $f_n$ is a smooth family of curves over an irreducible variety $T_n$ so that there exists a point $t\in T_n$ so
that $F_{n,t}:\mathcal{C}_{n,t}\to X$ is the normalisation morphism of $C_n$ composed with the inclusion, and $\dim T_n
= \dim \phi(T_n) = g$ where $\phi_n$ is the moduli map.

We give first an idea of the proof of this theorem. As mentioned, its proof relies on the existence of
infinitely many rational curves on a K3 surface and not on the full statement of \cite[Theorem
A]{regenerationinfinite}, so provides a new proof and a strengthening of the higher genus case of loc.\ cit.,
both in that $C_n^2 \to\infty$ implies $HC_n \to \infty$ by the Hodge Index Theorem, but also that the curves produced
vary in moduli. 

The second key ingredient in proving the above theorem is the logarithmic Bogomolov--Miyaoka--Yau inequality, which
allows us, using local analysis of Orevkov--Zaidenberg which we expand on in Section \ref{sec:maxmod}, to control the
singularities of rational curves in $X$ as their self-intersection increases. In particular, we show first in
Proposition \ref{CTBK3PROPCUSPRC} that if \[C^2>4690\] for $C$ a rational curve on a K3 surface, then $C$ must have a
locally reducible singularity (i.e., one with at least two branches). As it is not known whether such a rational curve
always exists on a K3 surface, we also show in Proposition \ref{CTBK3PROPINTSC} that if $C_1, C_2$ are two rational
curves so that $C_1C_2$ is large enough with respect to $C_1^2, C_2^2$, then they  must meet in at least two distinct
points (e.g., if $C^2\leq4690$ for all rational curves in the K3, then $C_1C_2> 1299546$ suffices). As a consequence, a
partial normalisation of such a $C$ or of such a union $C_1\cup C_2$ may now be deformed in $\bcalM_1(X,\beta)$ to
produce a genus one curve which necessarily deforms with maximal moduli. The argument then proceeds by induction on the
genus.

By results of Mukai, the general curve of genus $g$ is contained in a K3 surface if and only if $2\leq g\leq9$ or
$g=11$. Our result above however says that for any fixed K3 surface $X$ and any $g\geq0$, there exist $g$-dimensional
subvarieties of $\calM_g$ whose general member parametrises a curve which admits a morphism to $X$ birational onto its
image. In the opposite direction, it is worth noting that it is expected yet not known that a very general K3 surface
cannot be dominated by the product of two curves, which would imply that curves of constant moduli should not exist on
most K3 surfaces.

As far as applications are concerned, although the existence of rational curves is satisfying to know, they do not
provide much to work with. It turns out that the existence of one single genus 1 curve produced by Theorem \ref{thmA}
has numerous applications, so we begin by stating it as a separate corollary.

\begin{Corollaryx}
A K3 surface in characteristic zero contains a non-isotrivial family of integral curves of geometric genus 1.
\end{Corollaryx}

It is well-known that any K3 surface contains a family of genus 1 curves, so what is new in the above is the
variation in moduli. As an application, combined with a result of Voisin \cite[Theorem
11.1]{huybrechtsconstant} (where the existence of curves produced by the corollary is implicitly asked), the above
immediately implies a conjecture of Huybrechts \cite[Conjecture 2.3]{huybrechtsconstant}. 

\begin{Corollaryx}
There are infinitely many constant cycle curves of bounded order on every complex K3 surface $X$ and
their union is dense in the strong topology.
\end{Corollaryx}

In a different direction, even though $\HH^0(X,\Omega^1_X)=0$ is easy to see for a complex K3 surface $X$ via Hodge
theory, Kobayashi \cite[Corollary 8]{Kobayashi} also proved that a simply connected Calabi--Yau manifold has no
symmetric differentials, or in other words that \[\HH^0(X, \Sym^n\Omega^1_X)=0 \text{ for any }n>0.\] His proof is also
analytic in nature and relies on the resolution of the Calabi Conjecture by Yau. We give an algebraic proof of
this fact for K3 surfaces, using only the existence of one non-isotrivial family of genus 1 curves, which follows from
the Corollary above.

\begin{thmB}[Kobayashi]
The cotangent bundle of a complex K3 surface is not $\QQ$-effective. 
\end{thmB}

Based on his generalised Zariski decomposition, Nakayama in \cite{Nakayama} proved that this implies that the divisor
$\OO_{\PP(\Omega^1_X)}(1)$ is not even pseudoeffective (see Theorem \ref{thm:nakayama} for a proof).

Even though we do not provide a proof of Kobayashi's Theorem or Theorem A in positive characteristic, we state as many
results as possible in that direction and in the final Section \ref{sec:stability} we prove a conditional vanishing of
global 1-forms (known by theorems of Rudakov--Shafarevich or Nygaard) and stability of the cotangent bundle
(which holds if $X$ is not uniruled but is known to fail otherwise).

\noindent \textbf{Notation.} Throughout this paper a \textit{K3 surface} will always be a smooth projective simply connected
surface with trivial canonical divisor over an algebraically closed field.

\begin{Acknowledgements}
The idea to use rational curves to prove Nakayama's Theorem had been suggested to the second author by Claire Voisin
during a talk on the subject and we would like to thank her for this insight. We would like to thank Adrian Langer for
pointing out and filling in some missing steps in the proof of Theorem \ref{thm:nakayama}. The first named author is
partially supported by the NSERC Discovery Grant 262265 whereas the second is supported by the ERC Consolidator Grant
681838 ``K3CRYSTAL''.
\end{Acknowledgements}

\section{Deformations and singular curves}

Let $A$ be an effective divisor on a complex K3 surface. We consider the moduli map 
\[\begin{tikzcd} V_{A,g}\ar{r} &\bcalM_g\end{tikzcd}\] 
where $\bcalM_g$ is the moduli space of stable curves of genus $g$ and $V_{A,g}$ is the Severi variety parametrising
integral curves in $|A|$ of geometric genus $g$. It is expected that this map is generically finite over its image for
``most'' divisors $A\in \Pic(X)$, and we call such variation in moduli \textit{maximal} (see Definition \ref{def:maxmod} for a
more rigorous definition).
The problem of existence of curves moving with maximal moduli has been studied by various authors for generic complex K3
surfaces (cf.\ \cite{FKPS2, KemenyCurves, CFGK2015}).

\begin{Definition}
    Let $\mathbf{k}$ be an algebraically closed field and $C$ an integral curve over $\mathbf{k}$. We say that a point
    $p\in C$ is a locally reducible singularity of $C$ if the formal completion $\widehat{\OO}_{C,p}$ of the stalk of $C$ at $p$ is
    not an integral domain. Equivalently $\nu^{-1}(p)$ consists of at least two distinct points under the
    normalisation $\nu: C^\nu\to C$ of $C$. Otherwise, we say that $C$ is locally irreducible at $p$. The number of
    local branches of $C$ at $p$ is the number of points in $\nu^{-1}(p)$.
\end{Definition}

The following is standard and is the main reason we are interested in such singularities.

\begin{Lemma}\label{lem:red sing}
    Let $p\in C$ be a locally reducible singularity of an integral curve. Then the normalisation $\nu:C^\nu\to C$
    factors through a curve $C'$ which has one node and is smooth otherwise.
\end{Lemma}
\begin{proof}
Choose a sufficiently ample line bundle $L$ on $C$. For $q_1\neq q_2\in \nu^{-1}(p)$, consider the subspace
\[
V = \nu^* \HH^0(L) + \HH^0(\nu^* L \otimes \OO_{C^\nu}(-q_1-q_2)) \subset \HH^0(\nu^* L).
\]
Then $s_1(q_1) = s_2(q_2)$ for all $s_1,s_2\in V$. Let $f: C^\nu \to G \subset \PP V^*$ be the morphism given by the
linear series $V$. Clearly, $\nu$ factors through $f$. For $L$ sufficiently ample, $G$ has a node $q = f(q_1) = f(q_2)$
over $p$ as the only singularity.
\end{proof}

For $C\subset X$ a curve on a K3 surface, we denote by 
\[ \bcalM_g(X,\OO_X(C))\] 
the Kontsevich moduli space of stable maps of arithmetic genus $g$ to $X$ with image of class $\OO(C)$.
For $f:D\to X$ such a morphism, we denote by $[f]$ the induced point in moduli.

\begin{Definition}\label{def:maxmod}
Let $f:C\to X$ be a stable map of arithmetic genus $g$ to a K3 surface over an
algebraically closed field. We say that $f$ deforms
\begin{enumerate}
    \item \textit{in the expected dimension} if $\dim M=g$ for every irreducible component $[f]\in
    M\subset\bcalM_g(X,\OO(C))$ and
    \item \textit{with maximal moduli} if the induced moduli map $\phi_M:M\to\bcalM_g$ satisfies $\dim(
    \operatorname{im}\phi_M) \geq g$ for at least one irreducible component $[f]\in M\subset\bcalM_g(X,\OO(C))$.
\end{enumerate}
We say that an integral curve $C\subset X$ satisfies one of the above properties if its normalisation morphism
$\nu:C^\nu\to C\subset X$ composed with the embedding into $X$ does so.
\end{Definition}

\begin{Remark}
From \cite[Theorem 2.11]{regenerationinfinite}, for any $C\subset X$ integral with normalisation morphism contained in
some irreducible component $[\nu:C^\nu\to X]\in M\subset \bcalM_g(X,\OO(C))$, we have $\dim M\geq g$. Moreover, in
characteristic zero any such $C$ deforms in the expected dimension (from Proposition \ref{prop:ac}
below) but it is not necessarily the case that $C$ deforms with maximal moduli, as seen for example by the existence of
isotrivial elliptic fibrations. In positive characteristic the situation is more complicated, as on a uniruled K3 there
exist genus 0 curves which deform too much. Nodal rational curves on a K3 surface are always rigid though, and on a non-uniruled K3
surface every curve of geometric genus 1 deforms in the expected dimension (see \cite[Proposition
2.9]{regenerationinfinite}). We do not know any examples of curves that do not deform in the expected dimension on a
non-uniruled K3 surface.
\end{Remark}

The following is basically the Arbarello--Cornalba Lemma (see \cite[Lemma 1.4]{arbcorn} or \cite[\S XXI.9]{acgh} for a
more thorough reference) in the case of K3 surfaces.

\begin{Proposition}\label{prop:ac}
Let $X$ be a K3 surface over an algebraically closed field of characteristic zero, and $C\subset X$ be an integral curve of
geometric genus $g\geq1$. Then if $[\nu]\in M\subset\bcalM_g(X, \OO(C))$ is an irreducible component containing the
normalisation $\nu:C^\nu\to C$, we have
\begin{enumerate}
    \item A general element $[f:D\to X]\in M$ corresponds to an unramified morphism.
    \item $\dim M = g$.
    \item If $D'\subset X$ an integral curve and $[f:D\to X]\in M$ general, then the support of $f^*\OO_X(D')$ consists
    of $D'f(D)$ distinct points.
\end{enumerate}
\end{Proposition}
\begin{proof}
    The first claim is an application of the usual Arbarello--Cornalba Lemma in the case of K3 surfaces (see, e.g.,
    \cite{dedieusernesi}), whereas the second and third follow essentially from the first (see \cite[\S
    2]{regenerationinfinite} and the proof of \cite[Lemma 6.3]{regenerationinfinite}).
\end{proof}

\begin{Remark}
    In positive characteristic, it is not the case that (1) in the above is true (e.g., in a quasi-elliptic fibration
    the general fibre has ramified normalisation as it is a cusp), but we expect it to be true in most cases (see
    Question \ref{q:ac}). It is however true that (1) implies (2) and (3).
\end{Remark}

We recall the following argument, essentially due to Bogomolov--Mumford, cf.\ \cite[\S 13.2.1]{huybrechts}.

\begin{Proposition}\label{prop:maxmod1}
Let $X$ be a K3 surface over an algebraically closed field and $C\subset X$ be an integral curve of
geometric genus $g$. Assume further that $C$
\begin{enumerate} 
    \item deforms in the expected dimension,
    \item deforms with maximal moduli, and
    \item has a locally reducible singularity at a point $p$.
\end{enumerate} Then $C$ deforms to an integral curve $D$ of geometric genus $g+1$ which deforms in the expected
dimension and with maximal moduli.
\end{Proposition}
\begin{proof}
As the singularity at $p$ is locally reducible, from Lemma \ref{lem:red sing} we may take $f:\widetilde{C}\to X$ to be a
partial normalisation of $C$ which has one node over the point $p$ and is smooth otherwise. In particular
$[f]\in\bcalM_{g+1}(X, \OO(C))$.  Let $M$ be an irreducible component of $\bcalM_{g+1}(X,\OO(C))$ containing $[f]$.
From \cite[Theorem 2.11]{regenerationinfinite}, $\dim M\geq g+1$.  Consider now the moduli map
\[\begin{tikzcd} \phi: \bcalM_{g+1}(X,\OO(C))\ar{r}& \bcalM_{g+1}.\end{tikzcd}\]
Let $D_M$ be an irreducible component of $M\cap \phi^{-1}(\partial\bcalM_{g+1})$ containing $[f]$,
where $\partial\bcalM_{g+1} = \bcalM_{g+1} - \calM_{g+1}$ is the boundary divisor of $\bcalM_{g+1}$.

For a general point $[h]\in D_M$, $h: \Gamma\to X$ is a stable map such that $\Gamma$ is an integral curve of geometric
genus $g$ with a node and $h(\Gamma)$ and $C$ lie on the same component of $V_{C,g}$. Since $C$ deforms in the expected
dimension, $\dim D_M \le g$ and hence $D_M \subsetneq M$. On the other hand, since $\partial\bcalM_{g+1}$ is a
$\QQ$-Cartier divisor, $D_M$ has codimension one in $M$. We must have
\[
g + 1\le \dim M = \dim D_M + 1 \le g + 1
\]
and hence $\dim M = g+1$.
This proves that
for a general point $h:\Gamma \to X$ of $M$, $D = h(\Gamma)$ is an integral curve of geometric genus $g+1$ that deforms in the expected dimension.

Since $C$ deforms with maximal moduli, there exists an irreducible component $D$ of $
\phi^{-1}(\partial\bcalM_{g+1})$ containing $[f]$ such that $\dim \phi(D) = g$. Let $M'$ be an irreducible component of
$\bcalM_{g+1}(X,\OO(C))$ containing $D$. Since $\phi(M')$ is not contained in $\partial\bcalM_{g+1}$, we conclude
\[
g + 1 = \dim M' \ge \dim \phi(M') \ge \dim \phi(D) +1 = g+1
\]
and hence $\dim \phi(M') = g+1$. Therefore, for a general point $h:\Gamma \to X$ of $M'$, $D = h(\Gamma)$ is an integral
curve of geometric genus $g+1$ that deforms with maximal moduli.
\end{proof}

Although we will not be using it in this paper, we include the following immediate corollary, which is well-known to
experts, as an application.

\begin{Corollary}\label{cor:defofnodalrtlcurve}
Let $X$ be a K3 surface over an algebraically closed field and $R\subset X$ be a nodal rational
curve of arithmetic genus $g\geq1$. For any $1\leq d\leq g$, $R$ deforms to a nodal integral curve $C$ of geometric genus $d$
which deforms in the expected dimension and with maximal moduli.
\end{Corollary}
\begin{proof}
The result follows by induction, Proposition \ref{prop:maxmod1} and the fact that a general deformation of a nodal curve
will be nodal and as such has unramified normalisation morphism, hence deforms in the expected dimension from
\cite[Proposition 2.9]{regenerationinfinite}.
\end{proof}

One similarly obtains the following.

\begin{Proposition}\label{prop:maxmod2}
Let $X$ be a K3 surface over an algebraically closed field and $C_1,C_2\subset X$ be two integral
curves of geometric genus $g_1, g_2$ respectively. Assume further that 
\begin{enumerate} 
    \item $C_i$ deforms in the expected dimension for $i=1,2$,
    \item $C_i$ deforms with maximal moduli for $i=1,2$,
    \item $|C_1\cap C_2|$ contains at least two distinct points.
\end{enumerate}
Then $C_1\cup C_2$ deforms to an integral curve $D$ of geometric genus $g_1+g_2+1$ which deforms in the expected
dimension and with maximal moduli.
\end{Proposition}

\section{Families of curves of maximal moduli}\label{sec:maxmod}

There are two main ingredients in the proof of Theorem A 
\begin{itemize}
\item the existence of infinitely many rational curves on every complex K3 surface \cite{regenerationinfinite},
\item the logarithmic Bogomolov--Miyaoka--Yau (BMY) inequality \cite{LOGMY}. 
\end{itemize}

Let us first review the basics of the latter. For the applications that we have in mind, we start with a reduced but
possibly reducible curve $D$ on a smooth projective surface $X$ over $\CC$. Take now a log resolution
\[\begin{tikzcd} (\widehat{X}, \widehat{D})\ar{r} & (X,D),\end{tikzcd}\] 
i.e., a birational projective morphism $f: \widehat{X}\to X$ such that the total transform $\widehat{D} = f^{-1}(D)=
\sum_{i=1}^n \Gamma_i$ of $D$ has simple normal crossings, with irreducible components $\Gamma_i$ and $X\backslash D
\cong \widehat{X}\backslash \widehat{D}$. We usually choose $(\widehat{X}, \widehat{D})$ to be the minimal resolution of
$(X,D)$.

Now, for such a pair $(\widehat{X}, \widehat{D})$ of a smooth projective surface and a SNC divisor, the log BMY
inequality says that if $K_{\widehat{X}} + \widehat{D}$ is $\QQ$-effective, then
\begin{equation}\label{CTBK3E003}
(K_{\widehat{X}} + \widehat{D})^2 \le 3 c_2(\Omega^1_{\widehat{X}}(\log \widehat{D})).
\end{equation}
We recall that $\Omega^1_{\widehat{X}}(\log \widehat{D})$ is the locally free sheaf which sits in the following short
exact sequence
\[
\begin{tikzcd}
0 \ar{r} & \displaystyle{\Omega^1_{\widehat{X}}}
\ar{r} & \displaystyle{\Omega^1_{\widehat{X}}(\log \widehat{D})} \ar{r} & \displaystyle{\bigoplus_{i=1}^n
\OO_{\Gamma_i}}
\ar{r} & 0
\end{tikzcd}
\]
and we refer for example to \cite[\S 2]{esnaultviehweg} for further details.
\begin{Remark}
    Note that there is a version of the log BMY inequality over fields of positive characteristic, proven recently by
    Langer \cite{LangerLogBMYCharP}. The conclusion is essentially the same inequality, however one requires that the
    pair $(\widehat{X},\widehat{D})$ lifts in a compatible way to $W_2(k)$.
\end{Remark}

Over the complex numbers we have
\begin{equation}\label{CTBK3E004}
c_2(\Omega^1_{\widehat{X}}(\log \widehat{D})) = 
e(\widehat{X} \backslash \widehat{D}) = 
e(X \backslash D) = e(X) - e(D),
\end{equation}
where $e(\bullet)$ is the topological Euler characteristic. 

For the applications we have in mind, $X$ will be a K3 surface and hence $K_{\widehat{X}} + \widehat{D}$ will always be effective. 

Although $c_2(\Omega^1_{\widehat{X}}(\log \widehat{D}))$ can be computed topologically by \eqref{CTBK3E004} over
$\CC$, we want to give a purely algebraic formula for it in terms of $c_2(X)$, $p_a(D)$ and the invariants of the
singularities of $D$ (we refer to \cite[\S 5]{djp} for the basics of curve singularities). As the proof of this works in arbitrary
characteristic we state it in this generality. 

\begin{Lemma}\label{CTBK3LEMC2LOGDIFF}
Let $X$ be a smooth projective surface over an algebraically closed field and $D$ be a reduced curve on $X$. Let $(\widehat{X}, \widehat{D})$ be the minimal log resolution of $(X,D)$. Then
\begin{equation}\label{CTBK3E001}
c_2(\Omega^1_{\widehat{X}}(\log \widehat{D}))
= c_2(X) + (K_X + D) D - \sum_{p\in D} (2\delta_p - \gamma_p + 1)
\end{equation}
where $\delta_p$ and $\gamma_p$ are the $\delta$-invariant and the number of local branches of $D$ at $p$, respectively.
\end{Lemma}

\begin{proof}
Let $\widehat{D} = \sum_{i=1}^n \Gamma_i$, where $\Gamma_i$ are the irreducible components of $\widehat{D}$. From the exact sequences
\[
\begin{tikzcd}
0 \ar{r} & \displaystyle{\Omega^1_{\widehat{X}}(\log \sum_{i=1}^{m-1} \Gamma_i)}
\ar{r} & \displaystyle{\Omega^1_{\widehat{X}}(\log \sum_{i=1}^m \Gamma_i)} \ar{r} & \OO_{\Gamma_m}
\ar{r} & 0
\end{tikzcd}
\]
for $m=1,\ldots,n$, we obtain
\[
\begin{aligned}
\ch(\Omega^1_{\widehat{X}}(\log \widehat{D})) &= \ch(\Omega^1_{\widehat{X}})
+ \sum_{m=1}^n \ch(\OO_{\Gamma_m})\\
&= \ch(\Omega^1_{\widehat{X}})
+ \sum_{m=1}^n (\ch(\OO_{\widehat{X}})-\ch(\OO_{\widehat{X}}(-\Gamma_m)))\\
&= K_{\widehat{X}} + \widehat{D} + \frac{1}2(K_{\widehat{X}}^2 - 2c_2(\widehat{X}) - \sum_{m=1}^n \Gamma_m^2)
\end{aligned}
\]
where $\ch(\bullet)$ is the Chern character. It follows that
\[
\begin{aligned}
c_2(\Omega^1_{\widehat{X}}(\log \widehat{D})) &= c_2(\widehat{X})
+ \frac{1}2(K_{\widehat{X}} + \widehat{D})^2 - \frac{1}2 K_{\widehat{X}}^2 + \frac{1}2\sum_{m=1}^n \Gamma_m^2\\
&= c_2(\widehat{X}) + (K_{\widehat{X}} + \widehat{D}) \widehat{D}
- \sum_{1\le i<j\le n} \Gamma_i \Gamma_j.
\end{aligned}
\]
Note that further blowing up $\widehat{X}$ at a singularity of $\widehat{D}$ does not change
$c_2(\Omega^1_{\widehat{X}}(\log \widehat{D}))$. The minimal log resolution of $(X,D)$ does not blow up all singularities
of $D$ in case that $D$ is reducible: if $D$ has an ordinary double point at $p$ where two components of $D$ meet
transversely, we do not need to blow up $X$ at $p$. On the other hand, we can choose to blow up $X$ at such $p$ since it
does not change $c_2(\Omega^1_{\widehat{X}}(\log \widehat{D}))$. This has the advantage of streamlining our argument.
Hence we choose a log resolution $(\widehat{X}, \widehat{D})$ of $(X,D)$ which is minimal with the properties that
$\widehat{D}$ has simple normal crossings and the proper transforms of the components of $D$ are disjoint from each
other.

Let us write
\[
\widehat{D} = \sum_{i=1}^n \Gamma_i = \Delta + \sum_{p\in D_s} E_p
\]
where $\Delta$ is the proper transform of $D$ under $\pi: \widehat{X}\to X$
and $E_p = \pi^{-1}(p)$ for $p\in D_s$, where $D_s$ is the set of singularities of $D$. Clearly, $E_p$ is a tree of smooth rational curves
for all $p\in D_s$. 
Then the above equality takes the form
\[
\begin{aligned}
c_2(\Omega^1_{\widehat{X}}(\log \widehat{D})) &= c_2(\widehat{X})
+ (K_{\widehat{X}} + \Delta) \Delta + \sum_{p\in D_s}
(K_{\widehat{X}} + E_p) E_p + \sum_{p\in D_s} \Delta E_p\\
&\quad - \sum_{p\in D_s} \sum_{\substack{1\le i<j\le n\\
\Gamma_i\cup \Gamma_j\subset E_p}} \Gamma_i \Gamma_j.
\end{aligned}
\]
Since $\Delta$ is the normalisation of $D$,
\[
\begin{aligned}
(K_{\widehat{X}} + \Delta) \Delta &= 2p_a(\Delta) - 2 = 2p_a(D) - 2 - 2 \sum_{p\in D} \delta_p\\
&= (K_X + D) D - 2 \sum_{p\in D} \delta_p.
\end{aligned}
\]
For every $p\in D_s$, $p_a(E_p) = 0$ and hence
\[
\sum_{p\in D}
(K_{\widehat{X}} + E_p) E_p = -2 \sum_{p\in D_s} 1.
\]
It is also clear that $\Delta E_p$ equals the number of local branches of $D$ at $p\in D_s$. Therefore
\[
\sum_{p\in D_s} \Delta E_p = \sum_{p\in D_s} \gamma_p.
\]
Since $E_p$ is a tree of smooth rational curves,
\[
\sum_{\substack{1\le i<j\le n\\
\Gamma_i\cup \Gamma_j\subset E_p}} \Gamma_i \Gamma_j = |E_p| - 1
\]
for $p\in D_s$, where $|E_p|$ is the number of irreducible components of $E_p$. Finally,
\[
c_2(\widehat{X}) = c_2(X) + \sum_{p\in D_s} |E_p|.
\]
Combining all the above, we obtain \eqref{CTBK3E001}.
\end{proof}

For convenience, we write
\[
\mu_p = 2\delta_p - \gamma_p + 1.
\]
Over the complex numbers, $\mu_p$ agrees with the Milnor number of $D$ at $p$ (see \cite[Theorem 10.5]{milnor}). However
this can fail in positive characteristic, so we will call $\mu_p$ the {\em pseudo-Milnor number} of $D$ at $p$.

We now work towards constructing a lower bound for $(K_{\widehat{X}} + \widehat{D})^2$ in terms of $(K_X+D)^2$ and the
local contribution of the singularities of $D$. The following lemma is basically due to Orevkov--Zaidenberg \cite[\S
4]{ONSPC}, but we give here a simple proof that works in all characteristics.

\begin{Lemma}\label{CTBK3LEMOZ}
Let $X$ be a smooth projective surface over an algebraically closed field and $D$ be a reduced curve on $X$. Let $(\widehat{X}, \widehat{D})$ be the minimal log resolution of $(X,D)$. Then
\begin{equation}\label{CTBK3E005}
(K_{\widehat{X}} + \widehat{D})^2 \ge (K_X+D)^2 -
\sum_{p\in D} \left(1 - \frac{1}{m_p}\right)\mu_p
\end{equation}
where $m_p$ and $\mu_p$ are the multiplicity and pseudo-Milnor number of $D$ at $p$, respectively.
\end{Lemma}

\begin{proof}
As in the proof of Lemma \ref{CTBK3LEMC2LOGDIFF}, further blowing up $\widehat{X}$ at a singularity of $\widehat{D}$ does not change
$(K_{\widehat{X}} + \widehat{D})^2$. So we choose a log resolution $(\widehat{X}, \widehat{D})$ of $(X,D)$ which is minimal with the properties that $\widehat{D}$ has simple normal crossings and the proper transforms of the components of $D$ are disjoint from each other.

The proof of Lemma \ref{CTBK3LEMC2LOGDIFF} already gives
\begin{equation}\label{CTBKE009}
(K_{\widehat{X}} + \widehat{D}) \widehat{D} = (K_X + D) D -
\sum_{p\in D_s} \mu_p + \sum_{p\in D_s} (\gamma_p - 1).
\end{equation}
From now on we denote $K_{\widehat{X}/X}=K_{\widehat{X}}-\pi^*K_X$. Then, \eqref{CTBKE009} and the fact that
$K_{\widehat{X}/X}.\pi^*F=0$ for any divisor $F$ on $X$ yield
\[
\begin{aligned}
(K_{\widehat{X}} + \widehat{D})^2 - (K_X + D)^2 &= - \sum_{p\in D_s} \mu_p
+ \sum_{p\in D_s} (\gamma_p - 1) + (K_{\widehat{X}}^2 - K_X^2)\\
&\quad + \sum_{p\in D_s} K_{\widehat{X}} E_p + (K_{\widehat{X}} \Delta - K_X D)\\
&= - \sum_{p\in D_s} \mu_p
+ \sum_{p\in D_s} (\gamma_p - 1) + K_{\widehat{X}/X}^2\\
&\quad + \sum_{p\in D_s} K_{\widehat{X}} E_p + K_{\widehat{X}} (\Delta - \pi^* D).
\end{aligned}
\]
Thus, \eqref{CTBK3E005} holds as long as we can prove
\begin{equation}\label{CTBK3E014}
(\gamma_p - 1) + K_{\widehat{X}} E_p + (K_{\widehat{X}/X}^2)_p + (K_{\widehat{X}} (\Delta - \pi^* D))_p \ge \frac{\mu_p}{m_p} 
\end{equation}
for all $p\in D_s$. The problem is local so we work in a formal neighbourhood of a point $p\in D_s$ in $X$. For
simplicity, we drop the subscript $p$ in all notation so that $m = m_p$, $\mu = \mu_p$, $\gamma = \gamma_p$ and $E =
E_p$.

We can factor $\pi: \widehat{X}\to X$ into a sequence of blowups:
\[
\begin{tikzcd}[column sep=40pt]
\widehat{X} = X_a \ar{r}{\pi_{a,a-1}} & X_{a-1} \ar{r}{\pi_{a-1,a-2}}
& \ldots \ar{r} & X_1 \ar{r}{\pi_{1,0}} & X_0 = X
\end{tikzcd}
\]
where each $\pi_{i,i-1}: X_i\to X_{i-1}$ is the blowup of $X_{i-1}$ at one point for $i=1,2,\ldots,a$. Let $\pi_{i,j} = \pi_{j+1,j} \circ \pi_{j+2,j+1}
\circ \ldots \circ \pi_{i,i-1}$ be the birational map $X_i\to X_j$ for $0\le j < i\le a$ and let $F_i$ be the exceptional divisor of $\pi_{i,i-1}: X_i\to X_{i-1}$ for $i=1,2,\ldots,a$.
Then
\[
\begin{aligned}
K_{\widehat{X}/X} &= \pi_{a,1}^* F_1 + \pi_{a,2}^* F_2 + \ldots + \pi_{a,a-1}^* F_{a-1} + F_a\\
\Delta &= \pi^* D - m_1 \pi_{a,1}^* F_1 - m_2 \pi_{a,2}^* F_2 - \ldots - m_{a-1}\pi_{a,a-1}^* F_{a-1} - m_a F_a
\end{aligned}
\]
for some $m_i\in \ZZ_+$ satisfying that
\[
m = m_1 = \max_{1\le i \le a} m_i.
\]
It follows (see, e.g., \cite[Theorem 5.4.13]{djp}) that
\[
\begin{aligned}
\mu + \gamma - 1 = 2\delta &= \sum_{i=1}^a m_i(m_i - 1)\\
K_{\widehat{X}/X}^2 + K_{\widehat{X}} (\Delta - \pi^* D) &= \sum_{i=1}^a (m_i - 1).
\end{aligned}
\]
Therefore, \eqref{CTBK3E014} holds provided that we can prove
\begin{equation}\label{CTBK3E015}
(\gamma - 1) + K_{\widehat{X}} E\ge 0. 
\end{equation}

We recall that $E = \pi^{-1}(p)$ is a tree of smooth rational curves. Thus from the adjunction formula,
\[
K_{\widehat{X}} E = (K_{\widehat{X}} + E) E - E^2 = -2 - E^2 \ge -1
\]
where $E^2\le -1$ because the components of $E$ have negative definite intersection matrix. So we have \eqref{CTBK3E015} if $\gamma \ge 2$. Otherwise, $\gamma = 1$, i.e., $D$ has a locally irreducible or unibranch singularity at $p$. We claim that $K_{\widehat{X}} E \ge 0$ in this case.

Let $E_i = \pi_{i,0}^{-1}(p)$ for $i=1,2,\ldots,a$. Then $E_1 = F_1$ and
$K_{X_1} E_1 = -1$. If $\pi_{i,i-1}: X_i \to X_{i-1}$ is the blowup of $X_{i-1}$ at a smooth point of $E_{i-1}$, then
\[
E_i = \pi_{i,i-1}^* E_{i-1} \text{ and }
K_{X_i} E_i = K_{X_{i-1}} E_{i-1}.
\]
Otherwise, if $\pi_{i,i-1}: X_i \to X_{i-1}$ is the blowup of $X_{i-1}$ at a singular point of $E_{i-1}$, then
\[
E_i = \pi_{i,i-1}^* E_{i-1} - F_i \text{ and }
K_{X_i} E_i = K_{X_{i-1}} E_{i-1} + 1.
\]
In conclusion, we have
\[
K_{X_1} E_1 = -1 \text{ and }
K_{X_i} E_i = \begin{cases}
K_{X_{i-1}} E_{i-1} & \text{if } \pi_{i,i-1}(F_i) \not\in (E_{i-1})_\text{sing}\\
K_{X_{i-1}} E_{i-1} + 1 & \text{if } \pi_{i,i-1}(F_i) \in (E_{i-1})_\text{sing}
\end{cases}
\]
for $2\le i\le a$. Therefore, $K_{\widehat{X}} E = K_{X_a} E_a \ge 0$ as long as one of $\pi_{i,i-1}$ is the blowup of
$X_{i-1}$ at a singular point of $E_{i-1}$. For a locally irreducible singularity $p\in D_s$, it is easy to see that
$\pi_{a,a-1}: X_a \to X_{a-1}$ blows up $X_{a-1}$ at a singular point of $E_{a-1}$. Consequently $K_{\widehat{X}} E
\ge 0$ when $\gamma = 1$. This proves \eqref{CTBK3E015} and hence \eqref{CTBK3E014}, giving \eqref{CTBK3E005}.
\end{proof}

Combining \eqref{CTBK3E003}, \eqref{CTBK3E001} and \eqref{CTBK3E005}, we obtain
\begin{equation}\label{CTBK3E006}
(K_X+D)^2 -
\sum_{p\in D} \left(1 - \frac{1}{m_p}\right)\mu_p
\le 3 \big(c_2(X) + (K_X + D) D - \sum_{p\in D} \mu_p\big).
\end{equation}

We are now in a position to put all the above together for K3 surfaces in the characteristic zero case, where the
BMY inequality holds.

\begin{Proposition}\label{CTBK3PROPCUSPRC}
Let $D\subset X$ be an integral curve of geometric genus $g$ in a K3 surface over an algebraically closed field of
characteristic zero. If
\[
D^2 > 4690 + 550g + 16g^2,
\]
then $D$ has at least one locally reducible singularity.
\end{Proposition}

\begin{proof}
Suppose that $D$ only has locally irreducible singularities. Then
\begin{equation}\label{CTBK3E008}
(K_X + D) D - \sum_{p\in D} \mu_p 
= (K_X + D) D - 2\sum_{p\in D} \delta_p = 2g - 2.
\end{equation}
By \eqref{CTBK3E006} and ${\rm c}_2(X)=24$, we have
\begin{equation}\label{CTBK3E007}
D^2 - \sum_{p\in D} \left(1 - \frac{1}{m_p}\right) \mu_p \le 66 + 6g.
\end{equation}
Combining \eqref{CTBK3E008} and \eqref{CTBK3E007}, we have
\begin{equation}\label{CTBK3E009}
\sum_{p\in D} \frac{\mu_p}{m_p}\le 68 + 4g.
\end{equation}
On the other hand,
\begin{equation}\label{CTBK3E010}
\mu_p \ge m_p(m_p - 1)
\end{equation}
for all $p\in D$.
Putting \eqref{CTBK3E008}-\eqref{CTBK3E010} together gives
\[
\begin{aligned}
68 + 4g\ge \sum_{p\in D} \frac{\mu_p}{m_p} \ge \sum_{p\in D}
\left(
\sqrt{\mu_p + \frac{1}4} - \frac{1}2
\right) \ge \sqrt{D^2 + \frac{9}4 - 2g} - \frac{1}2
\end{aligned}
\]
where we note that we have used that the function $f(x)=\sqrt{x+\frac14}-\frac12$ vanishes at 0 and has everywhere negative second
derivative, hence is concave and $\sum f(x_i) \geq f(\sum x_i)$ for positive real $x_i$. It follows that $D^2 \le 4690 +
550g + 16g^2$. Therefore, $D$ has at least one locally reducible singularity if $D^2 > 4690 + 550g + 16g^2$.
\end{proof}

\begin{Proposition}\label{CTBK3PROPINTSC}
Let $D_1, D_2\subset X$ be two distinct integral curves in a K3 surface $X$ over an algebraically closed field of
characteristic 0. If
\[
2 D_1D_2 > \left(\sqrt{4D_1^2 + 9} + \sqrt{4D_2^2 + 9} + 2\right)(37 + D_1^2 + D_2^2) + 1, 
\]
then $D_1$ and $D_2$ meet at (at least) two distinct points.
\end{Proposition}

\begin{proof}
Suppose that $D_1$ and $D_2$ meet at a unique point $q$. Applying \eqref{CTBK3E006} to $(X, D = D_1 + D_2)$, we have
\begin{equation}\label{CTBK3E011}
D^2 - \sum_{p\in D} \left(1 - \frac{1}{m_{D,p}}\right)\mu_p\le 72 + 3(D^2 - \sum_{p\in D} \mu_p)
\end{equation}
where we use $\mu_{C,p}$ and $m_{C,p}$ to denote the pseudo-Milnor number and multiplicity of a reduced curve $C$ at $p$, respectively.

Note the following simple facts for $i=1,2$, $p\in D$ and $D_1\cap D_2=\{q\}$ as above
\begin{equation}\label{CTBK3E012}
\begin{aligned}
\mu_{D,p} &= \mu_{D_1,p} + \mu_{D_2,p} + 2 (D_1.D_2)_p - 1 \\
        &= \mu_{D_1,p} + \mu_{D_2,p} -1 \text{ if }p\neq q\\
\mu_{D,q} &= \mu_{D_1,q} + \mu_{D_2,q} + 2D_1D_2 - 1\\
m_{D,p} &= m_{D_1,p} + m_{D_2,p} \le M:= \sqrt{D_1^2 + \frac{9}4} + \sqrt{D_2^2 + \frac94} + 1.
\end{aligned}
\end{equation}
Combining \eqref{CTBK3E011} and \eqref{CTBK3E012}, we obtain
\[
\begin{aligned}
75 - 3\sum_{i=1}^2 \sum_{p\in D_i} \mu_{D_i,p}
=&\ 72 + 3 (D^2 - \sum_{p\in D} \mu_p) - 3(D_1^2 + D_2^2)\\
\ge&\ D^2 - \sum_{p\in D} \left(1 - \frac{1}{m_{D,p}}\right)\mu_{D,p}
- 3(D_1^2 + D_2^2)\\
\ge&\ 2(D_1D_2 - D_1^2 - D_2^2) - \sum_{p\in D} \left(1 - \frac{1}M\right)\mu_{D,p}\\
=&\ \frac{2}M D_1D_2 - 2(D_1^2 + D_2^2) +
\frac{M-1}{M}\left(1-\sum_{i=1}^2 \sum_{p\in D_i} \mu_{D_i,p}\right)
\end{aligned}
\]
Hence
\[
75 \ge \frac{2}M D_1D_2 + \frac{M-1}{M} - 2(D_1^2 + D_2^2)
\]
and the proposition follows.
\end{proof}

The lower bounds in the above propositions are almost certainly not optimal. Better results can be achieved with improvement to \eqref{CTBK3E005}
(cf. \cite{MOEONCMN}).

We are now ready to prove Theorem A.

\begin{proof}[Proof of Theorem A]
Let us first prove it for $g = 1$.

By \cite[Theorem A]{regenerationinfinite}, there are infinitely many integral rational curves $C_n$ on $X$. Suppose that
$C_n^2$ is unbounded. Then $C_n$ has a locally reducible singularity by Proposition \ref{CTBK3PROPCUSPRC} for $C_n^2$
sufficiently large. Such $C_n$ can be deformed to a non-isotrivial family of curves of geometric genus 1 by Proposition \ref{prop:maxmod1}.

Suppose that $C_n^2\le c$ for all $n$. We claim that
\begin{equation}\label{CTBK3E016}
\varlimsup_{\min(m,n)\to \infty} C_m C_n = \infty.
\end{equation}
Fixing $N\in \ZZ_+$, since $\rank_\ZZ\Pic(X)\le 20$, $C_N,C_{N+1},\ldots,C_{N+20}$ are linearly dependent in
$\Pic(X)_\QQ$. Suppose that
\begin{equation}\label{CTBK3E013}
a_0 C_N + a_1 C_{N+1} + \ldots + a_{20} C_{N+20} = 0
\end{equation}
in $\Pic(X)$ for some integers $a_i$, not all zero. Since $C_i$ are effective, $a_i$ cannot be all positive or negative. Let us rewrite \eqref{CTBK3E013} as
\[
F = \sum_{a_i>0} a_i C_{N+i} = - \sum_{a_j < 0} a_j C_{N+j}.
\]
Since $C_N, C_{N+1}, \ldots, C_{N+20}$ are distinct integral curves, it is easy to see that $F$ is nef. This implies
that there are
only finitely many integral rational curves $R$ such that $FR = 0$, since if $F^2=0$ then $F$ can only be zero on the (up to 24)
singular fibres of the elliptic fibration induced by $F$, and if $F^2>0$ then from the Hodge Index Theorem the
orthogonal space $F^\perp$ in the effective cone is negative definite and spanned by finitely many $-2$-curves. Hence there
exists $m\ge N$ such that $FC_m \ge 1$. Then $C_m + 2F$ is nef and big and hence
\[
\lim_{n\to\infty} (C_m+2F) C_n = \infty.
\]
Thus there exists $C\in \{C_N,C_{N+1},\ldots,C_{N+20}, C_m\}$ such that $C C_n$ is unbounded. This proves \eqref{CTBK3E016}.

By Proposition \ref{CTBK3PROPINTSC}, $C_m$ and $C_n$ meet at (at least) two distinct points for $C_m C_n$ sufficiently
large since $C_m^2\le c$ and $C_n^2 \le c$. There are infinitely many such pairs $C_m$ and $C_n$ by \eqref{CTBK3E016}
and
\[
\varlimsup_{\min(m,n)\to \infty} (C_m + C_n)^2 = \infty.
\]
Such $C_m\cup C_n$ can be deformed to a non-isotrivial family of curves of geometric genus 1 by Proposition
\ref{prop:maxmod2}, which as pointed out above will have unbounded self-intersection.  This proves the theorem for
$g=1$. The remaining cases follow from Propositions \ref{prop:maxmod1} and \ref{CTBK3PROPCUSPRC} by induction.
\end{proof}

\section{An algebraic Proof of Kobayashi's Theorem}

We say that a vector bundle $E$ on a quasi-projective variety $X$ is {\em $\QQ$-effective} if
\[
\HH^0(X, \Sym^m E) \ne 0
\]
for some positive integer $m$, where $\Sym^m E$ is the $m$-th symmetric product of $E$. We call $E$ {\em
pseudoeffective} if for every $n\in \ZZ_+$, there exists $m\in \ZZ_+$ such that
\[
\HH^0(X,\Sym^{mn} E \otimes \OO_X(mA))\ne 0,
\]
where $A$ is a fixed ample divisor on $X$. Alternatively, let
\[
Y = \PP(E^\vee) = \Proj(\Sym^\bullet E) = \Proj \bigoplus_{m\ge 0} \Sym^m E
\]
be the projectivisation of $E^\vee$ and let $\OO_Y(1)$ be the tautological bundle of $Y$ over $X$. By the Leray spectral
sequence, the $\QQ$-effectivity (resp.\ pseudoeffectivity) of $E$ coincides with that of $\OO_Y(1)$.

Let now $X$ be a K3 surface and let $Y = \Proj(S^\bullet \Omega^1_X)$
with $L = \OO_Y(1)$ being the tautological bundle of $\pi: Y\to X$. The following follows easily from Hodge theory over
the complex numbers, whereas in positive characteristic is a theorem of Rudakov--Shafarevich \cite{RSVectorfields} (see also Nygaard
\cite{nygaard}).
\begin{Theorem}\label{thm:no 1-forms}
    Let $X$ be a K3 surface over an algebraically closed field. Then $\HH^0(X, \Omega^1_X)=0$.
\end{Theorem}

See Proposition \ref{prop:unramifiedno1forms} for a simple, conditional algebraic proof of the above. In what follows we will
give an algebraic proof of Kobayashi's Theorem (i.e., Theorem B of the introduction), by reducing it to the above. The proof
in fact works in arbitrary characteristic under the following, minimal assumption.
\begin{Hypothesis}\label{CTBK3E002}
There exists an unramified morphism $f:E\to X$ from a smooth genus 1 curve
which deforms in the expected dimension and with maximal moduli.
\end{Hypothesis}
In characteristic zero, Theorem A (in combination with Proposition \ref{prop:ac}) produces infinitely many such curves, whereas in positive
characteristic we are not able to produce such a curve, although in remarks after the proof we will give various cases
in which such a curve does exist. 

\begin{Theorem}\label{thm:kobayashi}
Let $X$ be a K3 surface over an algebraically closed field.  If we assume Hypothesis \eqref{CTBK3E002}, then
\[\HH^0(X, \Sym^m\Omega^1_X)=0 \text{ for } m\geq1.\]
\end{Theorem}
\begin{proof}
We maintain the notation for $Y,L$ from the beginning of this section. Suppose for a contradiction that $L$ is
$\QQ$-effective. Let $m$ be the smallest positive integer such that $mL$ is effective and let $G\in |mL|$. We write
\[
G = \sum b_i D_i
\]
where $D_i \in |a_i L + \pi^* F_i|$ are the irreducible components of $G$ for some $a_i \in \NN$ and some divisors
$F_i\in \Pic(X)$ and $b_i\in \ZZ_+$ is the multiplicity of $D_i$ in $G$. Since $mL=\sum a_i b_i L + \sum b_i
\pi^*F_i$, we obtain that 
\begin{align}\label{eq:sum=0}
\sum b_i F_i=0 \text{ in } \Pic(X).
\end{align}

Let $C\subset X$ be an integral curve of geometric genus 1 as given by Hypothesis \ref{CTBK3E002}. 
From the assumption, there exists an irreducible curve $B\subset |C|$ with $C$ as member and such that every curve
$\Gamma\in B$ is of geometric genus 1.

When $a_i = 0$, $F_i$ is necessarily effective and $CF_i \ge 0$. Note also that there exists at least one
$i$ such that $CF_i\le 0$ and $a_i > 0$ since otherwise, $CF_i > 0$ for all $a_i > 0$ and so $\sum CF_i > 0$, contradicting
\eqref{eq:sum=0}.

From now on we denote by $a=a_i$, $D = D_i$ and $F = F_i$ so that $a_i > 0$ and $CF_i \le 0$.

From the assumption, the general deformation of the normalisation of $C$ is an immersion. We
henceforth replace $C$ by a general member of $B$ and let $\nu: E = C^\nu\to X$ be its normalisation, i.e., we have that
$\nu^* \Omega^1_X \to \Omega^1_E$ is a surjection. As the kernel is torsion-free on a smooth curve, it is a line bundle,
and by taking determinants we see that it must be isomorphic to $(\Omega^1_E)^\vee\cong\OO_E$. This leads to the exact sequence
\begin{equation}\label{CTBK3E000}
\begin{tikzcd}
0 \ar{r} & {\mathscr N}_\nu^\vee \ar{r}\ar[equal]{d} & \nu^* \Omega^1_X \ar{r} & \Omega^1_E \ar{r} \ar{r}\ar[equal]{d} & 0\\
& \OO_E && \OO_E
\end{tikzcd}
\end{equation}
where ${\mathscr N}_\nu$ is the normal bundle of $\nu$. From our assumption and the following lemma, the above sequence
does not split.

\begin{Lemma}
Sequence
\eqref{CTBK3E000} splits if and only if $B$ parametrises an isotrivial family of elliptic curves.
\end{Lemma}
\begin{proof} 
If $f:\mathcal{C}\to B$ the family with $B$ a smooth projective curve and $E$ the generic fibre of $f$,
then a section $\Omega^1_E\to\nu^*\Omega^1_X$ also induces a splitting of 
\[\begin{tikzcd} 0\ar{r}& f^*\Omega^1_B|_U\ar{r}& \Omega^1_{\mathcal{C}}|_U\ar{r}& \Omega^1_f|_U\ar{r}& 0\end{tikzcd}\] 
on some open subset $U\subset B$. Dualising this sequence and pushing forward to $U$ we get a split sequence whose first
coboundary map in cohomology is the Kodaira--Spencer map. Hence this map is necessarily zero so the family over $U$ is isotrivial.
\end{proof}

Since $aL+F$ is effective, $\HH^0(S^a \Omega^1_X \otimes \OO_X(F)) \ne 0$ and as $C$ is a general member of a covering
family of curves on $X$, we see that
\[
\HH^0(E, S^a \nu^* \Omega^1_X \otimes \OO_E(\nu^* F)) \ne 0
\]
as otherwise a global section of $S^a \Omega^1_X \otimes \OO_X(F)$ would vanish everywhere.
By \eqref{CTBK3E000}, $S^a \nu^* \Omega^1_X \otimes \OO_E(\nu^* F)$ has a filtration
\[
0\subsetneq E_1\subsetneq E_2\subsetneq \cdots \subsetneq E_{a+1}:=S^a \nu^* \Omega^1_X \otimes \OO_E(\nu^* F)
\]
with graded pieces all isomorphic to $\OO_E(\nu^* F)$. If the global section $\OO_E\to E_{a+1}$ from above vanishes when mapped to
$E_{a+1}/E_a=\OO_E(\nu^* F)$, then it must induce a non-zero global section of $E_a$. By induction, one of the
quotients $E_i/E_{i-1}$ must have a non-zero global section and hence $\HH^0(\OO_E(\nu^* F)) \ne 0$. On the other hand,
$CF \le 0$ and $\deg \nu^* F \le 0$. So we necessarily have $\OO_E(\nu^* F) = \OO_E$.

This proves that for all $i$ satisfying $a_i > 0$ and $C F_i \le 0$ we have $\OO_E (\nu^* F_i) = \OO_E$ and hence $CF_i = 0$.
For the remaining $i$, we clearly have $CF_i \ge 0$. Therefore, we conclude that $C F_i = 0$ for all $i$ from
\eqref{eq:sum=0}. In summary, we have
\begin{itemize}
\item if $a_i > 0$, $\OO_E (\nu^* F_i) = \OO_E$;
\item if $a_i = 0$, $F_i$ is effective and $CF_i = 0$.
\end{itemize}

As exact sequence \eqref{CTBK3E000} does not split,
\begin{align}\label{eq:nosections}
h^0(E, S^n \nu^* \Omega^1_X) = 1
\end{align}
for all $n\in \ZZ_+$.

Work now again with a fixed $i$ so that $a_i>0$ as above, keeping the notation $D,F,a$. Since $D$ is reduced, $Y_p =
\pi^{-1}(p)$ meets $D$ transversely for $p\in X$ general and as $C$ is a general member of a covering family of curves
on $X$, also $Y_p = \pi^{-1}(p)$ meets $D$ transversely for $p\in C$ general.  Let now $R = E\times_X Y \cong
\Proj(S^\bullet (\nu^* \Omega^1_X))$ with diagram
\[
\begin{tikzcd}
R \ar{r}{\rho} \ar{d} & Y\ar{d}{\pi}\\
E \ar{r}{\nu} & X.
\end{tikzcd}
\]
Since $Y_p$ and $D$ meet transversely for $p\in C$ general, $R_q$ and $\rho^* D$ meet transversely for $q\in E$ general,
where $R_q$ is the fibre of $R$ over $q$.

Note that $\rho^* D$ is a section of $a \rho^* L$. From \eqref{eq:nosections}, $h^0(R, n \rho^*
L) = 1$ for all $n\geq0$ and so we must have $\rho^* D = a \Gamma$, where $\Gamma$ is the unique section of $\rho^* L$. Then we
must have $a = 1$ because $R_q$ and $\rho^* D$ meet transversely for $q\in E$ general.

Hence we have concluded that $a_i = 0$ or $1$ for all $i$. If there are two distinct components $D_i$ and $D_j$ of $G$
such that $a_i = a_j = 1$, then $\rho^* D_i = \rho^* D_j = \Gamma$. Therefore,
\[D_i \cap \pi^{-1}(C) = D_j\cap \pi^{-1}(C)\]
for $C\in B$ general and hence $D_i = D_j$. Consequently, $G$ has only one component $D_i$ with $a_i = 1$ and so we have
$\HH^0(\Omega^1_X \otimes \OO_X(F)) \ne 0$ for some $F\in \Pic(X)$ such that $-F=\sum F_i$ is effective. As
$\HH^0(\Omega^1_X \otimes \OO_X(F))\subset \HH^0(\Omega^1_X)$ we obtain a contradiction from the case $m=1$, namely
Theorem \ref{thm:no 1-forms}.
\end{proof}

In conclusion, we have proved that $\Omega^1_X$ is not $\QQ$-effective if Hypothesis \eqref{CTBK3E002} holds. This of
course is a consequence of Theorem A in characteristic zero, but in the following remark we outline various cases where
this is true in characteristic zero under far weaker assumptions than the existence of infinitely many rational curves
on $X$.

\begin{Remark}\label{rem:cases}
\begin{enumerate}
\item Recall that from Propositions \ref{prop:maxmod1}, \ref{prop:maxmod2}, the existence of either one rational
curve $C\subset X$ with a locally reducible singularity, or two distinct rational curves meeting in at least two
distinct points guarantee the existence of a non-isotrivial family of genus 1 curves in $X$.

\item More generally, we can produce a non-isotrivial family of genus 1 curves on $X$ if there are distinct rational
curves $C_1,\ldots,C_n \subset X$ and points $p_i\ne q_i\in C_i^\nu$ on their normalisations such that for all $1\leq i<n$
\[
\nu(p_{i}) = \nu(q_{i+1}) \text{ and } \nu(p_n) = \nu(q_1)
\]
where $\nu: \sqcup C_i^\nu\to X$ is the normalisation of $\cup C_i$. In this case, we can find a
stable map $f: \Gamma\to X$ such that $\Gamma = \cup \Gamma_i$, $\Gamma_i \cong C_i^\nu$, $f(\Gamma_i) = C_i$,
\[
|\Gamma_1\cap \Gamma_2| = \ldots = |\Gamma_n \cap \Gamma_1| = 1 \text{ and } \Gamma_i \cap \Gamma_j = \emptyset
\text{ otherwise}.
\]
\end{enumerate}
\end{Remark}

In positive characteristic, even though there exist rational curves which deform too much and without unramified
deformations (e.g., a quasi-elliptic fibration on a supersingular K3 surface),
a version of the Arbarello--Cornalba Lemma (Proposition \ref{prop:ac}) eludes us for the time being. One could ask the
following.
\begin{Question}\label{q:ac}
    Let $f:C\to X$ be a morphism from a smooth projective curve of genus $g\geq1$ to a K3 surface over an algebraically
    closed field. If $f$ deforms in the expected dimension, is a general deformation of $f$ unramified?
\end{Question}
Assuming the above and that all rational curves in Remark \ref{rem:cases} are rigid, Propositions \ref{prop:maxmod1},
\ref{prop:maxmod2} imply that the cases listed in Remark \ref{rem:cases} also provide a genus 1 curve satisfying the
properties of Hypothesis \ref{CTBK3E002}, and hence Kobayashi's Theorem holds.

\section{Global 1-forms and stability}\label{sec:stability}

As mentioned in the introduction and in the previous section (see Theorem \ref{thm:no 1-forms}), the proof that a K3
surface does not have any global 1-forms uses analytic techniques in characteristic zero (Hodge theory) and is rather
non-trivial in positive characteristic. In this section we gather some auxiliary results and questions, giving simple,
conditional algebraic proofs of the fact that for a K3 surface $X$ we have that $\HH^0(X, \Omega^1_X)=0$ and that
$\Omega^1_X$ is slope-stable (with respect to any ample divisor), using only the existence of special curves in $X$. 

\begin{Proposition}\label{prop:unramifiedno1forms}
    Let $X$ be a smooth projective variety of dimension $n$ over an algebraically closed field and $f:C\to X$ an
    unramified morphism from a smooth curve of genus $g > 1$ so that $f$ deforms in a family which dominates $X$ and
    varies with maximal moduli. Then $\HH^0(X, T_X)=0$.
\end{Proposition}
\begin{proof}
    Taking cohomology of the sequence
    \[\begin{tikzcd} 0\ar{r}& T_C\ar{r}& f^*T_X\ar{r}& N_f\ar{r}& 0,\end{tikzcd}\]
    the Kodaira--Spencer map $\HH^0(C,N_f)\to \HH^1(C, T_C)$ must be injective, as it is the induced differential to the
    moduli map and $C$ deforms with maximal moduli. This implies that $\HH^0(C, f^*T_X)=0$, but as $C$ deforms to cover
    $X$, we obtain the result.
\end{proof}

In the case of K3 surfaces the existence of such curves in characteristic zero is guaranteed by Theorem A, but the
current proof relies on the existence of infinitely many rational curves, whose proof in fact uses the vanishing of 1-forms in a number of ways.
The assumptions of the above do hold unconditionally for K3 surfaces in the cases listed in Remark \ref{rem:cases}.

We move now to the question of stability of the (co)tangent bundle. We recall that for an ample divisor $A\in\Pic(X)$ on
a projective variety $X$ we say that a vector bundle $E$ on $X$ is \textit{$\mu_A$-(semi)stable} (often just $\mu$) if
\[
\mu_A(F):=\frac{\det(F)A^{\dim X-1}}{\operatorname{rk}(F)} {< \atop (\leq)} \frac{\det(E)A^{\dim X-1}}{\operatorname{rk(E)}} 
\]
for all torsion-free subsheaves $F\subsetneq E$. In fact if $F$ does not satisfy the above inequality then we say that
$F$ \textit{destabilises $E$}, and we may assume that $F$ is a
sub-vector bundle with torsion-free quotient. In particular if for a K3 surface $X$, $E=\Omega^1_X$ is not
semistable, then there exists a destabilising line bundle $L\subset\Omega^1_X$, i.e., $LA\geq0$.

The assumption we will be making to give a quick proof of stability of the tangent bundle will be the following.

\begin{Question}\label{q:nef cone}
    Let $X$ be a K3 surface over an algebraically closed field. Is it true that for any ample divisor $D\in\Pic(X)_\QQ$
    there exist integral curves $E_1,\ldots, E_n\subset X$ of geometric genus 1 so that $D=\sum_{i=1}^n a_i E_i$ for
    $a_i\in\QQ_{\geq0}$? 
\end{Question}

\begin{Remark}
    We note that the above is known to be true in the following cases
    \begin{enumerate}
        \item The Picard rank of $X$ is $\leq2$ \cite[Corollary 7.3, Theorem 8.4]{regenerationinfinite},
        \item $X$ contains no smooth rational curves: in many such cases the effective cone is generated by smooth genus
        1 curves even (see \cite{kovacs}). For the rest (in particular the case where the cone is not polyhedral) one can use
        the fact that every nef divisor can be written as a sum of minimal nef divisors and that each such divisor
        is linearly equivalent to an integral curve of geometric genus 1 (see \cite[\S 3]{regenerationinfinite} for the
        definition and for this result).
    \end{enumerate}
\end{Remark}

We claim that the stability of $\Omega_X^1$ follows from a positive answer to Question \ref{q:nef cone} for K3 surfaces
$X$. In fact, we can prove a more general statement. For that purpose, let us recall some basic facts about
Harder--Narasimhan filtrations and the cone of curves.

Let $E$ be a vector bundle on a smooth projective variety $X$. We use the notation $\mu_{A,\max}(E)$ to denote that the
maximum of the slopes $\mu_A(F)$ for all subsheaves $F\subset E$ and some ample $A$, which we from now on suppress in
the notation. This number is given by the Harder--Narasimhan filtration
$$
E = E_0 \supsetneq E_1 \supsetneq ... \supsetneq E_m \supsetneq E_{m+1} = 0
$$
of $E$, where $F_i = E_i/E_{i+1}$ are torsion-free and semistable
sheaves satisfying 
$$
\mu(F_0) < \mu(F_1) < ... < \mu(F_m)
$$
and $\mu_{\max}(E)$ is given by $\mu(F_m) = \mu(E_m)$.
Using Harder--Narasimhan filtrations, we have
$$
r \mu_{\max}(E) \ge \mu_{\max}(\wedge^r E)
$$
for all $1\le r\le \rank(E)$.

For a smooth projective variety $X$, we let $N_1(X)$ denote the group of $1$-cycles modulo numerical equivalence and let
$N_1(X)_\QQ$ and $N_1(X)_\RR$ denote $N_1(X)\otimes \QQ$ and $N_1(X)\otimes \RR$, respectively. For $X$ over $\CC$,
we have
$$
N_1(X)_\QQ \cong H^{n-1,n-1}(X,\QQ) = H^{n-1,n-1}(X)\cap H^{2n-2}(X,\QQ).
$$
For lack of a better term, we call the classes $A_1A_2\ldots A_{n-1}\in N_1(X)$ for ample $A_1,A_2,\ldots,A_{n-1}\in \Pic(X)$
{\em ample complete intersection classes}. We call the cone $\Amp_1(X)_\RR\subset N_1(X)_\RR$ generated by these classes
the {\em cone of ample complete intersection curves}.

\begin{Theorem}\label{thm:stab}
Let $X$ be a smoooth projective variety of dimension $n$ over an algebraically closed field of characteristic $0$ and
let $G\subset N_1(X)_\RR$ be the set consisting of numerical classes $\xi$ with the following property: there exists a sequence $f_m: C_m\to X$ of morphisms from smooth projective curves $C_m$ to $X$ such that
\begin{itemize}
\item $f_m(C_m)$ passes through a general point of $X$, i.e., the deformation of $f_m$ dominates $X$ for each $m$,
\item the numerical classes
$[(f_m)_* C_m]$ of $(f_m)_* C_m$ satisfy
$$
\lim_{m\to\infty} \frac{[(f_m)_* C_m]}{\deg (f_m)_* C_m} = \xi
$$
\item and the conormal bundles
$$
M_{f_m} = \ker(f_m^*\Omega^1_X \xrightarrow{} \Omega^1_{C_m})
$$
of $f_m$ satisfy
$$
\varlimsup_{m\to\infty} \frac{n\max (\mu_{\max}(M_{f_m}), \deg K_{C_m}) - \deg f_m^* K_X}{n \deg (f_m)_* C_m} \le 0
$$
\end{itemize}
where $\deg (f_m)_* C_m$ is the degree of $(f_m)_* C_m$ with respect to a fixed ample line bundle on $X$.

If $\Amp_1(X)_\RR$ is asymptotically generated by $G$, i.e., $\Amp_1(X)_\RR$ is contained in the closure of the cone generated by $G$, then
$\Omega^1_X$ is $\mu$-semistable for all ample divisors $A$ on $X$. More precisely, 
if $\Omega^1_X$ contains a locally free subsheaf $E$ of rank $r$ such that $\mu(E) \ge \mu(\Omega^1_X)$, then $n c_1(E) - r K_X$ is numerically trivial.

In particular, if $X$ is a complex K3 surface, $A$ is an ample divisor on $X$ and there is a positive answer to Question
\ref{q:nef cone}, then $\Omega^1_X$ is $\mu_A$-stable.
\end{Theorem}
\begin{proof}
Suppose that there exists a locally free subsheaf $E\subset \Omega^1_X$ of rank $r$ such that $\mu(E) \ge \mu(\Omega^1_X)$. Then
$L = \wedge^r E$ is a subsheaf of $\Omega_X^r$ and hence
$H^0(\Omega_X^r(-L)) \ne 0$.

Let $\xi\in G$ and $f_m: C_m\to X$ be the sequence of morphisms associated to $\xi$. Since $f_m(C_m)$ passes through a general point of $X$, we see that
$$
H^0(C_m, f_m^* \Omega_X^r(-L)) \ne 0.
$$
Then we have
$$
h^0(M_{f_m}^r (-f_m^* L) ) + h^0(M_{f_m}^{r-1}(-f_m^* L)
\otimes K_{C_m}) \ge 
h^0(f_m^* \Omega_X^r(-L)) > 0
$$
by the left exact sequence
$$
\begin{tikzcd}
0 \ar{r} & M_{f_m}^r \ar{r} & f_m^*\Omega_X^r \ar{r} & M_{f_m}^{r-1}
\otimes K_{C_m}
\end{tikzcd}  
$$
where $M_{f_m}^a = \wedge^a M_{f_m}$. On the other hand, we know that
$$
H^0(V(-B)) = 0 \text{ if } \deg B > \mu_{max}(V)
$$
for a vector bundle $V$ and a divisor $B$ on a smooth projective curve. It follows that
$$
\begin{aligned}
L . (f_m)_* C_m = \deg f_m^* L &\le \max\big(\mu_{max}(M_{f_m}^r),
\mu_{\max}(M_{f_m}^{r-1}) + \deg K_{C_m}\big)
\\
&\le \max\big(r \mu_{\max}(M_{f_m}),
(r-1)\mu_{\max}(M_{f_m}) + \deg K_{C_m}\big)\\
&\le r \max(\mu_{\max}(M_{f_m}), \deg K_{C_m}).
\end{aligned}
$$
Therefore,
$$
\left(\frac{L}r - \frac{K_X}{n}\right)
\frac{(f_m)_* C_m}{\deg (f_m)_* C_m} \le \frac{n\max (\mu_{\max}(M_{f_m}), \deg K_{C_m}) - \deg f_m^* K_X}{n \deg (f_m)_* C_m}.
$$
By our definition of $G$, we conclude that
$$
\left(\frac{L}r - \frac{K_X}{n}\right) \xi \le 0
$$
for all $\xi\in G$. On the other hand, since $\mu(E)\ge \mu(\Omega^1_X)$,
$$
\left(\frac{L}r - \frac{K_X}{n}\right) A^{n-1} \ge 0.
$$

Fixing $\xi\in G$,
since $\Amp_1(X)_\RR$ is open in $N_1(X)_\RR$,
$$
A^{n-1} - t\xi\in \Amp_1(X)_\RR
$$
for some $t>0$ sufficiently small. Since $\Amp_1(X)_\RR$ is asymptotically generated by $G$,
$$
A^{n-1} - t\xi = \sum_{m=1}^\infty t_m \xi_m
$$
for some $t_m > 0$ and $\xi_m\in G$. Finally, from
$$
(nL - rK_X)A^{n-1}\ge 0,\ (nL - rK_X)\xi \le 0
\text{ and }
(nL - rK_X)\xi_m \le 0,
$$
we conclude that $(nL - rK_X)\xi = 0$. Therefore,
$$
(nL - rK_X)\xi = 0
$$
for all $\xi\in G$. This implies that $nL - rK_X$ is numerically trivial since $G$ also generates $N_1(X)_\RR$.

For a complex K3 surface $X$, it is easy to see that $E/\deg E\in G$ for every elliptic curve $E$ on $X$. Since by
hypothesis the elliptic curves generate the ample cone $\Amp(X)$ of $X$, $\Amp(X)$ is generated by $G$. If $\Omega_X^1$
is destabilised by a line bundle $L$, then $L$ is numerically trivial. For K3 surfaces, this implies that $L = \OO_X$,
so that $H^0(\Omega_X^1(-L)) = H^0(\Omega_X^1) \ne 0$ which is a contradiction.
\end{proof}

\begin{Remark}
In positive characteristic, Langer \cite[\S 4]{langer15} proves that K3 surfaces not dominated by $\PP^2$ have strongly
semistable cotangent bundle, and that K3 surfaces admitting a quasi-elliptic
fibration (e.g., unirational K3 surfaces in characteristic 2) do not have semistable cotangent bundle. If semistable,
then $\Omega^1_X$ must also be stable as $\HH^0(X, \Omega^1_X)=0$ is known for an arbitrary K3. We expect
Question \ref{q:nef cone} to still have a positive answer here though. In fact if one could furthermore assume that all
the genus $1$ curves generating the nef cone admit normalisations which deform to unramified morphisms (something which
does not occur for fibres of a quasi-elliptic fibrations), the above proof goes through. 
\end{Remark}

We conclude this section by giving the proof of Nakayama's Theorem in arbitrary characteristic. This proof is essentially
the same as in \cite[Theorem 7.8]{bdpp} (which draws from Nakayama's original proof from \cite{Nakayama}) with the necessary
adjustments for positive characteristic in place.

\begin{Theorem}[Nakayama in characteristic $p\geq0$]\label{thm:nakayama}
Let $X$ be a K3 surface over an algebraically closed field $k$. Assume further that $\Omega^1_X$ is $\mu$-stable and that
\[\HH^0(X, \Sym^n\Omega^1_X)=0\text{ for all }n>0.\]
Then $\Omega^1_X$ is not pseudoeffective.
\end{Theorem}
\begin{proof}
Since stability persists if we pass to a larger algebraically closed field, we may assume $k$ is uncountable. Let
$Y=\PP(\Omega^1_X)$ and suppose for a contradiction that $L=\OO_Y(1)$ is pseudoeffective. Then there is a
Nakayama--Zariski decomposition of $L=E+N$ where $E$ is an effective $\RR$-divisor and $N$ is nef in codimension 1 (due
to \cite{Nakayama} in characteristic 0 and \cite{mustata, fl} otherwise).

From \cite[Theorem 4.1]{langer10} (or Flenner or Mehta--Ramanathan's Theorem in characteristic zero), we may pick a very
ample smooth curve $C$ on $X$ so that $\Omega^1_X|_C$ is strongly semistable (or just semistable in characteristic
zero).  Then on the ruled surface $R=\PP(\Omega^1_X|_C)$ every pseudoeffective line bundle is nef (in fact for the
projectivisation of a degree zero strongly semistable bundle on a curve, these cones agree). On the other hand, $L|_R$
is not ample, since $L^2|_R=c_1(\Omega^1_X)\cdot C=0$. Hence $L|_R$ is on the boundary of the nef cone of $R$. Write
$E=aL+\pi^*E'$.  As the Picard number of $R$ is two and $E|_R$ is also $\RR$-effective, it must be that $E'.C\geq0$. If
$E'.C>0$ then $E'$ is effective on $X$ (as $C$ can vary), so in particular $E|_R$ is big and hence ample. This contradicts
$L|_R=E|_R+N|_R$ being boundary on the nef cone though. In other words, $CE'=0$ and as $C$ can vary, $E'=0$, forcing
$E=aL$. Then $a=0$, since from the assumption $L$ has no effective multiple. It follows that $E=0$ and $L$ is nef in
codimension 1. In particular it fails to be nef on at most countably many curves $C_i$. Taking a hyperplane section $H$
of $Y$, we see then that $L|_H$ is nef. In particular, $L^2\cdot H\ge 0$. In terms of Chern classes, this means that
\[ -c_2(T_X)\ge 0, \]
which contradicts $c_2(T_X)=24$.
\end{proof}

\end{document}